\newcommand{\notes}[1]{}
\newcommand{\vekk}[1]{}
\newcommand{\delete}[1]{}
\newcommand{\gtcomment}[1]{}

\RequirePackage{ifpdf} 

\ifpdf
\documentclass[american,12pt]{article}
\else
\documentclass[american,12pt]{article}
\fi

\usepackage{babel}

\usepackage[colorlinks=true,linkcolor=blue,citecolor=blue]{hyperref}
\usepackage{doi}

\usepackage{natbib}

\usepackage{color} 

\usepackage[latin1]{inputenc} 
\usepackage{latexsym}

\usepackage{amsmath}
\usepackage{amsfonts}
\usepackage{amssymb}

\usepackage{amsthm}
\usepackage{float}

\newtheorem{Def}{Definition}

\newtheorem{prop}{Proposition} 
\newtheorem{cor}{Corollary} 

\newcommand{\tP}{\mathop{\text{P}}}

\newcommand{\sX}{\mathcal{X}}
\newcommand{\sY}{\mathcal{Y}}
\newcommand{\sZ}{\mathcal{Z}}

\newcommand{\calF}{{\cal F}} 
\newcommand{\calG}{{\cal G}} 

\newcommand{\cmeasure}{C-measure}
 % ... 

\newcommand{\ceps}{{\cal E}} 
\newcommand{\cb}{{\cal B}} 
\newcommand{\ca}{{\cal A}} 
\newcommand{\cf}{{\cal F}} 

\newcommand{\pr}{\operatorname{\text{P}}} 

\newcommand{\st}{\operatornamewithlimits{{\mid}}}
 
\newcommand{\nc}{\newcommand}
\nc{\beqns}{\begin{eqnarray*}}
\nc{\eeqns}{\end{eqnarray*}}
\nc{\beqn}{\begin{eqnarray}}
\nc{\eeqn}{\end{eqnarray}}
\nc{\beq}{\begin{equation}}
\nc{\eeq}{\end{equation}}

\newcommand{\bes}{ \begin{equation} \begin{split} } 
\newcommand{\ees}{ \end{split} \end{equation} } 

\newcommand{\into}{\mbox{$\: \rightarrow \:$}}

\newcommand{\imply}{\Rightarrow}

\newcommand{\RealN}{{\mbox{$\Bbb R$}}} 

\ifpdf
\begin{document}

\else
\usepackage{graphicx}
\begin{document}
\DeclareGraphicsExtensions{.eps} 
\fi

\newcommand{\change}[1]{{\color{red} #1}}




\vekk{
\setlength{\oddsidemargin}{0in}
\setlength{\evensidemargin}{0in}
\setlength{\topmargin}{-.5in}
\setlength{\headsep}{0in}
\setlength{\textwidth}{6.5in}
\setlength{\textheight}{8.5in}
\def\refhg{\hangindent=20pt\hangafter=1}
\def\refmark{\par\vskip 2mm\noindent\refhg}
\def\refhg{\hangindent=20pt\hangafter=1}
\def\refmark{\par\vskip 2mm\noindent\refhg}
\def\refhg{\hangindent=20pt\hangafter=1}    
\def\refhgb{\hangindent=10pt\hangafter=1}
\def\refmark{\par\vskip 2mm\noindent\refhg}
\renewcommand{\baselinestretch}{1.5}
}

\title{Conditional probability and improper priors}
\author{Gunnar Taraldsen and Bo H. Lindqvist (2016)\\
  Communications in Statistics - Theory and Methods,\\ 45:17, 5007-5016,
  \doi{10.1080/03610926.2014.935430}}

\date{}

\maketitle

\tableofcontents

\newpage

\begin{abstract}
According to Jeffreys improper priors are needed to get the Bayesian machine up and running.
This may be disputed, but usage of improper priors flourish.
Arguments based on symmetry or information theoretic reference
analysis can be most convincing in concrete cases.
The foundations of statistics as usually formulated rely on 
the axioms of a probability space, 
or alternative information theoretic axioms that imply 
the axioms of a probability space.
These axioms do not include improper laws, 
but this is typically ignored in papers that consider improper priors.  

The purpose of this paper is to present a mathematical theory that can
be used as a foundation for statistics that include improper priors.
This theory includes improper laws in the initial axioms and has in
particular Bayes theorem as a consequence.
Another consequence is that some of the usual calculation 
rules are modified.
This is important in relation to common statistical practice which usually
include improper priors, 
but tends to use unaltered calculation rules.
In some cases the results are valid,
but in other cases inconsistencies may appear.
The famous marginalization paradoxes
exemplify this latter case.

An alternative mathematical theory for the foundations of statistics
can be formulated in terms of conditional probability spaces.
In this case the appearance of improper laws is a consequence of the theory.
It is proved here that the resulting mathematical structures for the
two theories are equivalent.
The conclusion is that the choice of the first or the second
formulation for the initial axioms can be considered a matter of
personal preference.
Readers that initially have concerns regarding improper priors can
possibly be more open toward a formulation of the initial axioms in
terms of conditional probabilities.    
The interpretation of an improper law is given by the corresponding
conditional probabilities.
\end{abstract}

\vspace*{.3in}


\vfill

\newpage

\section{Introduction}

Statistical developments are driven by applications, 
theory, 
and most importantly the interplay between applications and theory.
The following 
 is intended for readers that can appreciate 
the importance of the theoretical foundation for probability theory as
given by the axioms of \citet{KOLMOGOROV}.
According to the recipe of Kolmogorov a random element in a set $\Omega_X$ is
identified with a measurable function $X: \Omega \into \Omega_X$ where
$(\Omega, \ceps, \pr)$ is the basic probability space that the whole
theory is based upon.

The use of improper priors is common in the statistical literature,
but most often without reference to a corresponding theoretical foundation.
It will here be explained how the theory of conditional probability
spaces as developed by \citet{RENYI} is
related to a theory for statistics that includes improper priors.
This theory has been presented in a simplified form with some elementary examples by
\citet{TaraldsenLindqvist10ImproperPriors}.
The idea is to use the above recipe given by Kolmogorov,
but generalized by assuming that $(\Omega, \ceps, \pr)$ 
is defined by the use of a $\sigma$-finite measure.
The underlying law $\pr$ itself is not a $\sigma$-finite measure,
but it is an equivalence class of $\sigma$-finite measures.
A more precise formulation is given by Definition~\ref{dCondMeas} below.

In oral presentations of the theory related to improper priors it is
quite common that someone in the audience makes the following claim:
{\em Improper priors are just limits of proper priors, 
so we need not consider improper priors.
}
We strongly disagree with this even though it is true that
improper priors can be obtained as limits of proper priors.
The reason is perhaps best explained by analogy with a more familiar example:
It is true that the real number system is obtained as
a limit of the rational numbers.
A precise construction, and this is important, 
is found by the aid of equivalence classes of Cauchy sequences. 
Nonetheless, most people prefer to think of real numbers as such
without reference to rational numbers.
Real numbers are just limits of rational numbers,
but we use them with properties as given by the axioms of the real
number system.

The reader will not find any new algorithms or methods for the
solution of practical problems here.
The presented theory can, however, be used to put many known solutions
to practical problems on a more solid theoretical foundation.
This additional effort is necessary to avoid and explain 
contra-dictionary results as exemplified by for instance the famous 
marginalization paradoxes \citep{StoneDawid72,DawidStoneZidek73}. 
The theory also gives a natural frame 
for the proof of optimality of inference based on fiducial theory
\citep{TaraldsenLindqvist13fidopt} and
a proof of coincidence of a fiducial distribution 
and certain Bayesian posteriors based on improper priors \citep{TaraldsenLindqvist13fidpost}.
The theory has also been used for a rigorous specification of
intrinsic conditional autoregression models \citep{LavineHodges12icar}. 
These models are widely used in spatial statistics, dynamic linear
models, and elsewhere.


\section{Conditional measures}
\label{sCondMeas}

The aim in the following is to formulate a theory that can be used to 
provide a foundation for statistics that includes improper priors.
A less technical presentation of this with some elementary examples 
has already been provided by
\citet{TaraldsenLindqvist10ImproperPriors}.
They show in particular by examples that this theory is different
from the alternative theory for improper priors provided by \citet{HARTIGAN}.
This section gives a condensed presentation of the 
mathematical ingredients.
Most definitions are standard as
presented by for instance \citet{RUDIN},
but some are not standard and are emphasized in the text.
An example is the concept of a {\em \cmeasure} as introduced below. 

Let $\sX$ be a nonempty set, 
and $\calF$ a family of subsets that includes the
empty set $\emptyset$.
The family $\calF$ is a $\sigma$-field if it is closed
under formation of countable unions and complements.
A set is measurable if it belongs to $\calF$.
A measurable space $\sX$ is a set $\sX$ equipped with a 
$\sigma$-field $\calF$.
The same symbol $\sX$ is here used to denote both the set and the space.
The notation $(\sX, \calF)$ is also used to denote the measurable space. 
The convention here is to use the term {\em space} to
denote a set with some additional structure.

A measure space $(\sX, \calF, \mu)$ is a measurable space $(\sX, \calF)$ equipped with a measure $\mu$.
A measure is a function $\mu: \calF \into [0,\infty]$ 
that is countably additive:
The equality
%
$
\mu (\bigcup_i A_i) = \sum_i \mu (A_i)
$
%
holds for disjoint measurable $A_1, A_2, \ldots$.
\begin{Def}
Let $(\sX, \calF, \mu)$ be a measure space.
An admissible condition $A$ is a measurable set with
$0 < \mu A < \infty$.
\end{Def}
The measure space $(\sX,\calF,\mu)$ is $\sigma$-finite
if there exists a sequence $A_1, A_2, \ldots$ of admissible conditions such
that $\sX = \cup_i A_i$. 
A probability space is a 
measure space $(\sX, \calF, \mu)$ 
such that $\mu (\sX) = 1$,
and $\mu$ is then said to be a probability measure.

Consider the set $M$ of all $\sigma$-finite measures on a 
fixed measurable space $(\sX,\calF)$.
The set $M$ includes in particular all probability measures,
and the following gives new concepts also when restricted to
probability measures.
Two elements $\mu$ and $\nu$ in $M$ are defined to be
equivalent if $\mu = \alpha \nu$ for some positive number $\alpha$.
This defines an equivalence relation $\sim$ and the
quotient space $M / \sim$ defines the set of
{\em \cmeasure{}s}.
It should in particular be observed that any topology on $M$
induces a corresponding quotient topology on $M / \sim$.
Convergence of \cmeasure{}s is an important topic,
but the study of this is left for the future.
Some further discussion will, however, be provided
in Section~\ref{sRenyi}.
A {\em \cmeasure \ space} $(\sX,\calF,\gamma)$ is
a  measurable space $(\sX,\calF)$ equipped with a \cmeasure\ $\gamma$.
This means that 
$\gamma = [\mu] = \{\nu \st \nu \sim \mu\}$ is
an equivalence class of $\sigma$-finite measures.
The term {\em conditional probability space}
will here be used as an equivalent term.
This convention will be motivated next,
and is further elaborated in Section~\ref{sRenyi}.

Let $(\sX, \calF, \gamma)$ be a conditional probability space, 
and let $A$ be an admissible condition.
The conditional law,
or equivalently the {\em conditional measure}, 
$\gamma (\cdot \st A)$ 
is then well defined by $\gamma (B \st A) = \nu (B \st A) = \nu (B \cap A) / \nu(A)$
where $\nu \in \gamma$.
It is well defined since it does not depend on the choice of $\nu$,
and the resulting conditional law is a probability measure.
This argument gives:
\begin{prop}
A conditional probability space $(\sX,\calF,\gamma)$ defines a
unique family of probability spaces indexed by the admissible sets:
$\{(\sX, \calF, \gamma(\cdot \st A)) \st A \text{ is admissible}\}$.
\end{prop}
For ease of reference we restate the following definition.
\begin{Def}
\label{dCondMeas}
Let $(\sX,\calF)$ be a measurable space.
A conditional measure
$\gamma = [\nu] = \{\alpha \nu \st 0<\alpha<\infty, \nu \text{ is $\sigma$-finite}  \}$ 
is a set of $\sigma$-finite measures defined by a $\sigma$-finite
measure $\nu$ defined on $(\sX,\calF)$.
Let $A$ be an admissible condition and let $B$ be measurable. 
The formula $\gamma (B \st A) = \nu (B \cap A) / \nu (A)$ defines the
conditional probability of $B$ given $A$.
A conditional measure space $(\sX,\calF,\gamma)$ is
a  measurable space $(\sX,\calF)$ equipped with a conditional measure $\gamma$.
\end{Def}
In general probability and statistics it is most useful to extend the
definition of conditional probability and expectation to
include conditioning on $\sigma$-fields and statistics.
This can be done also in the more general context here.
The main new ingredient is
given by the definition
of $\sigma$-finite $\sigma$-fields and $\sigma$-finite measurable functions.


Let $(\sX, \cf, \mu)$ be a measure space.
Assume that $\cf_1 \subset \cf$ is a 
{\em $\sigma$-finite $\sigma$-field}
in the sense that
$(\sX, \cf_1, \mu_1)$ is $\sigma$-finite where
$\mu_1$ is the restriction of $\mu$ to $\cf_1$.
This implies that $(\sX, \cf, \mu)$ is also $\sigma$-finite.
Let $A \in \cf$. 
The conditional measure $\mu (\cdot \st \cf_1)$
is defined by the $\cf_1$ measurable function
$x \mapsto \mu (A \st \cf_1)(x)$
uniquely determined by the relation
\begin{equation}
\label{eqCondDefSigma}
\mu (A \cap B) = \int_B \mu (A \st \cf_1)(x)\, \mu_1 (dx) 
\end{equation}
which is required to hold for all measurable subsets $B \in \cf_1$. 

The existence and uniqueness proof follows by observing:
(i) $\nu B = \mu (A B)$ defines a measure on $\cf_1$.
(ii) $\nu$ is dominated by the measure $\mu_1$.
(iii) The Radon-Nikodym theorem 
gives existence and uniqueness of the conditional 
$\mu (A \st \cf_1)$ as the density of $\nu$ with respect to $\mu_1$
so the claim
$\nu (dx) = \mu (A \st \cf_1)(x) \mu_1 (dx)$ follows.
The uniqueness is only as a measurable function defined
on the measure space $(\sX, \cf_1, \mu_1)$
and the conditional probability is more properly 
identified with an equivalence class of measurable functions.

The defining equation~(\ref{eqCondDefSigma})
shows that $\mu (A \st \cf_1) = (\alpha \mu) (A \st \cf_1)$ 
for all $\alpha > 0$.
It can be concluded that the conditional measure 
$\gamma (A \st \cf_1)$ is well defined if
$(\sX, \cf, \gamma)$ is a conditional probability space.
An immediate consequence is
$\gamma (X \st \cf_1) = 1$, 
so the conditional measures are all normalized.  
The term {\em conditional probability} will motivated by this
be used as equivalent to the term {\em conditional measure}.
This is as above for the elementary conditional measure.

The following example demonstrates that this conditional probability 
generalizes the elementary conditional probabilities 
$\gamma (A \st B)$.
Let $\cf_1$ be the $\sigma$-field generated by 
a countable partition of  
$\sX$ into disjoint admissible sets $A_1, A_2, \ldots$.
It follows that $\cf_1$ is $\sigma$-finite.
Assuming this the conditional expectation is given by
$\gamma (A \st \cf_1) (x) = \gamma (A \st A_i)$ for $x \in A_i$.


The definitions presented so far lead naturally to the definition
of the category of conditional probability spaces with
a corresponding class of arrows.
The study of this, 
and functors to related categories, will not be pursued here.
This more general theory gives, however, 
alternative motivation
for some of the concepts presented next.

A function $\phi: \sX \into \sY$ is measurable if the inverse image of
every measurable set is measurable:
$(\phi \in A) = \phi^{-1} (A) = \{x \st \phi(x) \in A\}$ is
measurable for all measurable $A$.
Let $\mu$ be a measure on $\sX$.
The image measure $\mu_\phi$ is defined by
\begin{equation} 
\mu_\phi (A) = \mu (\phi \in A)
\end{equation}
{\em A measurable function $\phi$ is by definition $\sigma$-finite if $\mu_\phi$ is $\sigma$-finite.}
A direct verification shows
that a $\sigma$-finite function $\phi: \sX \into \sY$
pushes a conditional probability space structure on $\sX$ into a 
conditional probability space
structure on $\sY$. 
This follows from the above and the identity
$[\mu]_\phi = [\mu_\phi]$. 
Consequently, if $\gamma$ is a \cmeasure, then the
\cmeasure\ $\gamma_\phi$ is well defined if
$\phi$ is $\sigma$-finite.
The definition given here of a  $\sigma$-finite function
is a generalization of the concept of a 
regular random variable as defined by
\citet[p.73]{RENYI}. 
The definition of $\sigma$-finite $\sigma$-fields and $\sigma$-finite
measurable functions can be reformulated as follows.
\begin{Def}
Let $(\sX, \calF, \mu)$ be a $\sigma$-finite measure space.
A $\sigma$-field $\calF_1 \subset \calF$ is $\sigma$-finite
if $\mu$ restricted to $\calF_1$ is $\sigma$-finite.
Let $(\sY, \calG)$ be a measurable space.
A measurable
$\phi: \sX \into \sY$ 
is $\sigma$-finite if the $\sigma$-field
$\calF_1 = \{\{x \st \phi (x) \in A\} \st A \in \calG\}$ is $\sigma$-finite.
\end{Def}
The previous arguments show that the $\sigma$-finite functions
can play the role as arrows in the category of conditional probability spaces.
The $\sigma$-finite functions can also be used to
define conditional probabilities just as
$\sigma$-finite $\sigma$-fields did in
equation~(\ref{eqCondDefSigma}).
It will be a generalization since the previous definition is
obtained by consideration of the function 
$x \mapsto x$
as a function taking values in the space equipped with the 
$\sigma$-finite $\sigma$-field in the construction that follows.

Assume that $\delta: \sX \into \sZ$ is $\sigma$-finite.
The conditional probability $\mu^z (A) = \mu (A \st \delta = z)$ is defined by the relation
\begin{equation}
\label{eqCondDef}
\mu (A [\delta \in B]) = \int \mu (A \st \delta = z) [z \in B]\, \mu_\delta (dz) 
\end{equation}
which is required to hold for all measurable subsets $B \subset \sZ$. 
The existence and uniqueness proof follows by an argument
similar to the argument after equation~(\ref{eqCondDefSigma}).
The identity $[\mu]^z = \mu^z$ holds and 
{\em the conditional measure $\gamma^z$ is well defined for a \cmeasure\ $\gamma$}.

Composition of the functions
$x \mapsto \delta (x) = z$ and
$z \mapsto \mu^z (A)$ defines 
the conditional probability
$\mu (A \st \delta)$ as a measurable function
defined on $\sX$.
This function is measurable with respect to
the initial $\sigma$-field $\cf_\delta \subset \cf$ generated by $\delta$.
The $\sigma$-finiteness of $\delta$ is equivalent with the
$\sigma$-finiteness of $\cf_\delta$.
Direct verification shows that the definitions of 
conditional probability given by equations~(\ref{eqCondDefSigma}) and (\ref{eqCondDef})
coincide in the sense that
$\mu (A \st \cf_\delta) = \mu (A \st \delta)$. 

The conclusion is that a conditional probability space
$(\sX,\calF, \gamma)$ is not only equipped with the
family of elementary conditional probabilities 
$\{\gamma (\cdot \st A) \st A \in \ca\}$, 
but also a family
$\{\gamma^z \st z \in \sZ\}$ of conditional probabilities for each 
$\sigma$-finite $\delta: \sX \into \sZ$. 

{\em The conditional probability $\mu_\phi^z$ on $\sY$} is defined by
\begin{equation}
\label{eqCondProbY}
\mu_\phi^z (A) = \mu^z (\phi \in A) = \mu (\phi \in A \st \delta = z)
\end{equation}
The function $\delta$ must be $\sigma$-finite,
but it is not required that $\phi : \sX \into \sY$ is $\sigma$-finite.
It follows that
\begin{equation}
\label{eqCondFact}
\mu_{\phi,\delta} (dy,dz) = \mu_\phi^z (dy) \mu_\delta (dz)
\end{equation}
The case $\sX=\sY$ and $\phi (x) = x$ gives the defining
equation~(\ref{eqCondDef})
as a special case of the more general factorization given by
equation~(\ref{eqCondFact}).

The previous discussion can be summarized by:
\begin{prop}
A conditional measure space $(\sX,\calF,\gamma)$,
a $\sigma$-finite 
$\delta: \sX \into \sZ$,
and a measurable $\phi: \sX \into \sY$
define a unique measurable family 
$\{\gamma_\phi^z \st z \in \sZ\}$
of conditional measures
defined on $\sY$.
\end{prop}
It does not follow in general that there exists a version of
$\gamma_\phi^z$ such that this is a measure for almost all $z$.
A sufficient condition for this is that 
$\sY$ is a Borel space \citep[p.618]{SCHERVISH}.
This is the case if $\sY$ is in one-one measurable correspondence
with a measurable subset of an uncountable complete separable metric
space \citep[p.406]{ROYDEN}. 
The corresponding version of the conditional probability is
then said to be a {\em regular conditional probability}.
Integration with respect to $\gamma_\phi^z$ can nonetheless be defined
without any regularity conditions,
and the factorization given by equation~(\ref{eqCondFact}) holds in
the most general case as stated. 
The possibility of this more general integral with respect to
conditional probabilities was indicated already by
\citet[eq.10 on p.54]{KOLMOGOROV}.

\section{Statistics with Improper Priors}
\label{sStat}

Let $x$ be the observed result of an experiment.
It will be assumed that $x$ can be identified with
an element of a mathematically well defined set $\Omega_X$.
The set should include all other possible outcomes that could
have been observed as a result of the experiment.
The observed result can be a number, 
a collection of numbers organized in some way,
a continuous function,
a self-adjoint linear operator, 
a closed subset of a topological space,
or any other element of a well defined set corresponding
to the experiment under consideration. 

Assume that the sample space $\Omega_X$ is equipped
with a $\sigma$-field $\ceps_X$ so that
$(\Omega_X, \ceps_X)$ is a measurable space.
Assume furthermore that
$(\Omega_X, \ceps_X, \pr_X^\theta)$ is a 
probability space for each
$\theta$ in the model parameter space $\Omega_\Theta$.
The family 
$\{\pr_X^\theta \st \theta \in \Omega_\Theta\}$
specifies a statistical model for the experiment.

A predominant family of example in the applied statistical literature
is given by letting $\pr_X^\theta$ be the
multivariate Gaussian distribution on $\Omega_X = \RealN^N$
with covariance matrix $\Sigma (\theta)$ and mean $\mu (\theta)$
where $\Omega_\Theta \subset \RealN^K$.
The simplest special case is given by
$\Sigma = I$ and $\mu = (\theta; \ldots ; \theta)$ which
corresponds to independent sampling from the
univariate Gaussian with unknown mean $\theta$ and variance equal to $1$.
Other examples included are given by ANOVA models with fixed and
random effects, more general regression models,
structured equations models with latent variables
from the fields of psychology and economy \citep{Joreskog70sem}, 
and a variety of models from the statistical 
signal processing literature \citep{vanTrees02ArrayProcessing}.
These models correspond to specific choices of
the functional dependence on $\theta$ in 
$\Sigma (\theta)$ and $\mu (\theta)$.

The contents so far coincides with the definition
of a statistical model as found in standard statistical literature.
One exception is the choice of the notation
$\{\pr_X^\theta \st \theta \in \Omega_\Theta\}$
for the statistical model.
This choice indicates
the connection to the theory of conditional probability spaces
as will be explained by the introduction of further assumptions.

It is assumed that the statistical model is based upon an underlying
abstract conditional probability space $(\Omega, \ceps, \pr)$.
This includes the case of an underlying abstract probability space
as formulated by \citet{KOLMOGOROV} as a special case,
but seen as a conditional probability space.
It is abstract in the sense that it is assumed to exist, 
but it is not specified.
It is assumed that the model parameter space is a measurable
space $(\Omega_\Theta, \ceps_\Theta)$,
that there exists a $\sigma$-finite 
measurable $\Theta: \Omega \into \Omega_\Theta$,
and that there exists a measurable
$X: \Omega \into \Omega_X$
so that the resulting conditional probability
$\pr_X^\theta$ as defined in equation~(\ref{eqCondProbY})
coincides with the specified statistical model.

Existence of $\Omega$, $\pr$, $\Theta$, and $X$
can be proved in many concrete cases by consideration
of the product space $\Omega_X \times \Omega_\Theta$
equipped with the $\sigma$-finite measure 
$\pr_X^\theta (dx) \pi (d\theta)$
obtained from the choice of a $\sigma$-finite measure $\pi$.
This includes in particular the multivariate Gaussian example
indicated above.
As soon as existence is established it
is assumed for the further theoretical development
that  $\Omega$, $\pr$, $\Theta$, and $X$
are abstract unspecified objects with the required 
resulting statistical model as a consequence.

It can be observed that it is required that the mapping
$\theta \mapsto \pr_X^\theta (A)$
is measurable for all measurable $A$ for the
above construction to be possible.
This condition is trivially 
satisfied in most examples found in applications,
and is furthermore a typical assumption in theoretical developments.
A good example of the latter is given by the mathematical proof of the factorization theorem
for sufficient statistics \citep{HalmosSavage49}.

The basis for frequentist inference is then the observation $x$ and 
the specified statistical model $\pr_X^\theta$ based
on the underlying abstract conditional probability space 
$(\Omega,\ceps, \pr)$.

The basis for Bayesian inference is as for frequentist inference,
but the prior distribution $\pr_\Theta$ is also specified.
The basis for Bayesian inference is hence the observation $x$ 
and the joint distribution 
$\pr_{X,\Theta} (dx,d\theta) = \pr_X^\theta (dx) \pr_\Theta (d\theta)$. 
The conclusions of Bayesian inference are derived from
the posterior distribution
$\pr_\Theta^x$,
which is well defined by 
equation~(\ref{eqCondProbY})
if $X$ is $\sigma$-finite.
This result can be considered to be a very general version of 
{\bf Bayes theorem} as promised in the Abstract.
A discussion of a more elementary version involving densities is given
by \citet{TaraldsenLindqvist10ImproperPriors}.

The importance of the $\sigma$-finiteness of $X$ has also
been observed by others, 
but then as a requirement on the prior.  
\citet[p.911]{BergerBernardoSun09refprior} includes this requirement
as a part of the definition of a {\em permissible prior}.
The definition as formulated in this section can be 
used as a generalization of this part to cases
not defined by densities. 

A summary of the contents in this section is given by:


\begin{Def}[Statistical model]
\label{dStatMod}
A {statistical model} 
$\{(\Omega_X, \ceps_X, \pr_X^\theta) \st \theta \in \Omega_\Theta\}$
is specified by a family of probability spaces indexed by the model parameter space
$(\Omega_\Theta, \ceps_\Theta)$
with the additional structure defined in the following.

It is assumed that all objects are defined based on
the underlying conditional probability space $(\Omega, \ceps, \pr)$.
The observation is given by a measurable function
$X: \Omega \into \Omega_X$
and the model parameter is given by
a $\sigma$-finite measurable function
$\Theta: \Omega \into \Omega_\Theta$.
It is assumed that the family of probability measures
is given by the conditional law, so
$\pr_X^\theta (A) = \pr (X \in A \st \Theta = \theta)$.

A {Bayesian statistical model} is specified by
a statistical model together with
a specification of the prior law $\pr_\Theta$.
It is assumed that 
$X$ is $\sigma$-finite,
and then the resulting marginal law $\pr_X$
is a conditional measure 
and the resulting posterior law 
$\pr_\Theta^x (B) = \pr (\Theta \in B \st X=x)$ is well
defined.
\end{Def}

In the previous the prior $\pr_\Theta$,
the marginal $\pr_X$,
and the joint distribution $\pr_{X,\Theta}$
are \cmeasure s 
with corresponding 
conditional probability spaces
$(\Omega_\Theta, \ceps_\Theta, \pr_\Theta)$,
$(\Omega_X, \ceps_X, \pr_X)$, and
$(\Omega_{X,\Theta}, \ceps_{X,\Theta}, \pr_{X,\Theta})$.
The interpretation of the prior is in terms of the
corresponding elementary conditional laws $\pr_\Theta (\cdot \st A)$.
The same holds for the other improper laws.

Bayesian inference is essentially unique.
This is in contrast to frequentist inference which
most often offer many different possible inference procedures for a
given problem.
An analogous situation occurs in applied metrology where it is
common to have many different measurement instruments available
for the measurement of a physical quantity.
The choice of instrument depends on the actual situation and purpose
of the experiment at hand.  
\par

The previous gives a mathematical definition of 
a {\em statistical model} and
a {\em Bayesian statistical model} based on the concept of
a conditional measure.
The concept of a {\em fiducial statistical model} can also be defined based
on the same theory.
The necessary ingredients and further discussion of this have been 
presented by \citet{TaraldsenLindqvist13fidopt,TaraldsenLindqvist13fidpost}.

\section{Renyi Conditional Probability Spaces}
\label{sRenyi}

\label{sCondSpace}

\citet[p.38-]{RENYI} 
gives a definition of a conditional probability space
based on a family of objects $\mu (A \st B)$.
A condensed summary of the initial ingredients
in this theory is presented next,
but with some extensions and minor changes in the choice 
naming conventions.
The purpose is to show the close connection to the concept
of a conditional measure space as discussed in the previous two section.
The final words of Renyi on this subject are recommended
for a more thorough \citep{RENYI} and pedagogical
presentation \citep{RENYI_PROBABILITY} of the theory as formulated and
motivated by Renyi.


\begin{Def}[Bunch]
\label{dBunch}
Let $(\sX, \calF)$ be a measurable space.
A family $\cb \subset \calF$ is a bunch in $\sX$ if
\begin{enumerate}
\item $B_1, B_2 \in \cb \imply B_1 \cup B_2 \in \cb$.
\item There exist $B_1, B_2, \ldots \in \cb$ such that 
$\bigcup_i B_i = \sX$.
\item The empty set $\emptyset$ does not belong to $\cb$.
\end{enumerate}
\end{Def}

{\it Example 1} Let  $(\sX, \calF)$ be the real line equipped with the Borel
$\sigma$-field.
Let $\cb$ be the set of finite nonempty unions of open intervals on the form
$(m/2,1 + m/2)$ where $m$ are integers.
The family $\cb$ is then a countable bunch.
\hfill $\Box$
%
%


%
\begin{Def}[Renyi space]
\label{dRenyiSpace}
A Renyi space  $(\sX, \calF, \nu)$ is a measurable space
$(\sX, \calF)$ equipped with a family 
$\{\nu (\cdot \st B) \st {B \in \cb}\}$ 
of probability measures indexed by a bunch $\cb$ which fulfill   
$B_1, B_2 \in \cb$ and $B_1 \subset B_2$ 
$\imply \nu (B_1 \st B_2) > 0$, and the identity
\begin{equation}
\label{eqRenyiSpace}
\nu (A \st B_1) = \frac{\nu (A \cap B_1 \st B_2)}{\nu (B_1 \st B_2)}
\end{equation}
\end{Def}

A Renyi space  $(\sX, \calF, \nu_2)$ extends 
a Renyi space  $(\sX, \calF, \nu_1)$
by definition if $\cb_1 \subset \cb_2$ and
$\nu_1 (\cdot \st B) = \nu_2 (\cdot \st B)$ for all $B \in \cb_1$.
The extension is strict if $\cb_1 \subset \cb_2$ and $\cb_1 \neq \cb_2$.
A Renyi space is maximal if a strict extension does not exist.

{\it Example 1 (continued)}
Let $\nu (A \st B)$ be the uniform probability 
law on $B$ for each $B \in \cb$.
This gives a Renyi space $(\sX, \calF, \nu)$
where $(\sX, \calF)$ is the real line with the Borel 
$\sigma$-field.
The family 
$\{\nu (\cdot \st B)\}_{B \in \cb}$ 
is in this case a countable family of probability measures.

Let $\mu = [m]$ be the \cmeasure\ given by
Lebesgue measure $m$ on the real line.
The elementary conditional measures
$\mu (\cdot \st A)$ for admissible $A \in \ca$ defines
a Renyi space $(\sX, \calF, \mu)$
which contains the 
Renyi space $(\sX, \calF, \nu)$
in the sense that 
$\cb \subset \ca$ and
$\nu (\cdot \st B) = \mu (\cdot \st B)$
for all $B \in \cb$.
It follows from the results presented next
that $(\sX, \calF,\mu)$
is a maximal extension of
$(\sX, \calF,\nu)$.
\hfill $\Box$


It follows generally that
a \cmeasure\ space $(\sX, \calF, \mu)$ 
generates a unique 
Renyi space  $(\sX, \calF, \mu)$ through the
elementary conditional measures
$\mu (\cdot \st A)$.
The same symbol $\mu$ is here used for two different concepts. 
Further excuse for this abuse of notation 
is given by the following structure result: 
%
\begin{prop} 
\label{tRenyi}
A Renyi space generates a unique conditional measure space.
The corresponding resulting Renyi space is a maximal extension of the
initial Renyi space.
\end{prop}
\begin{proof}
Let $(\sX, \calF, \nu)$ be the Renyi space.
It will be proved that there exists a
$\sigma$-finite measure $\mu$
such that $\mu (\cdot \st B) = \nu (\cdot \st B)$
for all $B \in \cb$,
and that the
\cmeasure\ $[\mu]$ is unique.

The first step in the proof is to pick an arbitrary $B_0 \in \cb$ and
define $\mu (B) = \nu (B \st B_0 \cup B) / \nu (B_0 \st B_0 \cup B)$ 
for $B \in \cb$.
This choice gives the normalization $\mu (B_0) = 1$.
This definition is extended to measurable $A \subset B \in \cb$ by
$\mu (A) = \mu (A \cap B) = \nu (A \st B) \mu (B)$.
An arbitrary measurable $A$ can be written 
as a disjoint union of measurable $A_1, A_2, \ldots$
where each $A_i$ is contained in some set $B$ from the bunch.
The measure $\mu$ is then finally defined by
$\mu (A) = \sum_i \mu (A_i)$.

Equation~(\ref{eqRenyiSpace}) can be used to prove
that the previous definition of $\mu (A)$ based on
a $A \subset B$ for $B \in \cb$ does not depend on the choice of $B$.
This, 
and further proof of consistency and uniqueness of
$[\mu]$ is left to the reader.
An alternative is to consult the proof of a corresponding result 
given by \citet[p.40-43]{RENYI}.
\end{proof}
%
%

Two different Renyi spaces can generate the same \cmeasure\ space.
A concrete example is provided by consideration of the
bunch generated by the intervals $(m/3,m/3 +1)$ in addition
to the two bunches considered in Example 1.
It follows generally that the set of Renyi spaces based on a given measurable
space is strictly larger than the set of \cmeasure s on the
measurable space.

\begin{cor}
\label{cor1}
A Renyi space has a unique extension to a maximal Renyi space.
The set of maximal Renyi spaces is in one-one correspondence
with the set of \cmeasure\ spaces.
\end{cor}
\begin{proof}
Let  $(\sX, \calF, \nu)$ be a Renyi space and
let $(\sX, \calF, [\mu])$ be the corresponding 
generated \cmeasure\ space.
The Renyi space $(\sX, \calF, [\mu])$ given by the set of
admissible conditions $\ca$ is then a unique maximal extension.
Uniqueness and maximality follows since
any Renyi space that contains  $(\sX, \calF, \nu)$
will generate the \cmeasure\ space $(\sX, \calF, [\mu])$
by the construction given in the proof of Proposition~\ref{tRenyi}.
\end{proof}

A more general concept of a conditional probability space was
originally introduced by \citet{Renyi55axioms},
and a corresponding more general structure theorem 
was proved by \citet{Csaszar55struct}.
\citet[p.95]{RENYI} refers to these more general spaces as
generalized conditional probability spaces.
They are truly more general and
a generalized conditional probability space is not necessarily generated by a single
$\sigma$-finite measure.


\section{Discussion}
\label{sDiscussion}

The distinction between a $\sigma$-finite measure space and the
corresponding \cmeasure\ space could at first sight seem trivial.
For a $\sigma$-finite measure $\mu$ the corresponding
\cmeasure\ $\nu = [\mu]$ is an equivalence class of 
$\sigma$-finite measures in the set of all $\sigma$-finite measures on
the measurable space $\sX$.
It follows,
as stated earlier, 
that any topology on the set of $\sigma$-finite measures
gives a corresponding quotient topology on the set 
of \cmeasure s.
Convergence of $\sigma$-finite measures is different from
convergence of the corresponding conditional measures.
This is also true if the initial $\sigma$-finite measure is a
probability measure.

An alternative is to consider 
the \cmeasure\ space as a maximal Renyi space,
and this is a concept more clearly distinct from that of a 
$\sigma$-finite measure space.
Convergence concepts for Renyi spaces can be studied directly,
and initial work on this has been done by \citet{RENYI}.
He shows in particular that any conditional measure can be obtained
as a limit of conditional measures corresponding to probability
measures in a reasonable topology.
The study of convergence concepts exemplify an important
difference between $\sigma$-finite measures and
\cmeasure s.
This is left for the future.

The distinction between a $\sigma$-finite measure space and the
corresponding \cmeasure\ space 
can also be seen by analogy with the construction of projective spaces.
The projective space $\tP^n (\RealN)$ as a set is
the set of lines through the origin $0$ in $\RealN^{n+1}$.
It is hence equal to the set of
equivalence classes 
$[x] = \{\lambda x \st \lambda \in \RealN \setminus \{0\}, x \in
\RealN^{n+1}\setminus \{0\} \}$ in $\RealN^{n+1}\setminus \{0\}$.
The set of \cmeasure s on a measurable space
is hence different from the set of $\sigma$-finite measures just
as a projective space is different from the space on which it is constructed.

The presented theory is in line with the arguments given by \citet{JEFFREYS}.
He argues that improper priors are necessary to get the Bayesian
machine up and running.
This point of view can be disputed, 
but it is indisputable that usage of improper 
priors flourish in the statistical literature.  
There is hence a need for a theory that includes improper priors.

\citet{LINDLEY}, 
apparently in line with the view of the current authors,
found that the theory of Renyi is a natural starting
point for statistical theory.
In the Preface to his classical text on probability he writes:
\begin{quote} 
The axiomatic structure used here is not the usual one associated
with the name of Kolmogorov. Instead one based on the
ideas of Renyi has been used. The essential difference between
the two approaches is that Renyi's is stated in terms of conditional
probabilities, whereas Kolmogorov's is in terms of
absolute probabilities, and conditional probabilities are defined
in terms of them. Our treatment always refers to the probability
of A, given B, and not simply to the probability of A. In my
experience students benefit from having to think of probability
as a function of two arguments, A and B, right from the beginning.
The conditioning event, B, is then not easily forgotten
and misunderstandings are avoided. These ideas are particularly
important in Bayesian inference where one's views are
influenced by the changes in the conditioning event.
\end{quote}

\citet{LINDLEY} refers to an earlier 
German edition of
the book cited here \citep{RENYI62}.
The two books \citep{RENYI_PROBABILITY,RENYI}
represent the final view of Renyi regarding conditional probability
spaces,
but the basis for the theory development are found in earlier articles
\citep{RenyiTuran76selectedpapers,Renyi55axioms}.
The extension given by conditioning on $\sigma$-finite 
statistics and $\sigma$-finite $\sigma$-fields is not treated by Renyi.

\vekk{
As a final concrete example
consider a real random quantity $\Theta$
such that $\pr_\Theta (E \st B) = m (E \cap B) / m(B)$
where $m$ is Lebesgue measure on the real line.
The family $\cal B$ of sets $B$ can be taken to be the family
of finite nonempty unions of finite open intervals
with rational end points.
This is then a countable bunch of events as in Definition~\ref{dBunch}.
The interpretation of this is that if it is known that
$\Theta$ has values in a finite interval $(a,b)$,
then the corresponding conditional law is the uniform law on the interval.
This can be interpreted in a subjective or a frequentist sense
depending on the concrete problem at hand.
The structure theorem of Renyi formulated here as Proposition~\ref{tRenyi}
ensures that this countable family of conditional probability laws
can be uniquely extended so that it
corresponds to the \cmeasure\ law 
$\pr_\Theta = [m]$ given by the family of measures equivalent to
Lebesgue measure.  
The prior $\pr_\Theta$ is not equal to Lebesgue measure,
but to the equivalence class generated by Lebesgue measure.
This equivalence class was also obtained in Example 1 starting
from a much smaller Renyi space.
}

The structure theorem shows in general that 
a family of conditional probabilities that satisfies 
the axioms of a Renyi space given in Definition~\ref{dRenyiSpace}
can be extended 
so that it corresponds 
to a unique maximal Renyi space which 
can be identified
with a \cmeasure\ space.
The family of conditional probabilities
gives intuitive motivation and interpretation for usage of
improper laws in probability and statistics. 
In this theory any marginal law corresponds to a 
conditional probability space.
All probabilities are conditional probabilities.

\bibliography{bib,gtaralds}
\bibliographystyle{plainnat} 

\end{document}